\documentclass[a4paper,11pt,leqno]{amsart}
\usepackage{amsmath,amsthm,amssymb,latexsym,epsfig,graphicx,subfigure}
\numberwithin{equation}{section} \numberwithin{figure}{section}

\newtheorem{thm}{Theorem}[section]
\newtheorem{theorem}{Theorem}[section]
\newtheorem{lm}[thm]{Lemma}
\newtheorem{lemma}[thm]{Lemma}

\theoremstyle{definition}
\newtheorem{notation}[thm]{Notation}
\newtheorem{pr}[thm]{Proposition}
\theoremstyle{definition}
\newtheorem{df}[thm]{Definition}
\theoremstyle{definition}
\newtheorem{definition}[thm]{Definition}
\theoremstyle{definition}
\newtheorem{rem}[thm]{Remark}
\newtheorem{remark}[thm]{Remark}
\newcommand{\define}[1]{{\em #1\/}}

\newcommand{\cF}{{\mathcal F}}
\newcommand{\fg}{{\mathfrak g}}
\newcommand{\cH}{{\mathcal H}}
\newcommand{\cL}{{\mathcal L}}

\newcommand{\fv}{{\mathfrak v}}

\newcommand{\field}[1]{\mathbb{#1}}
\newcommand{\Sph}{\field{S}}   
\newcommand{\eps}{\epsilon}
\newcommand{\vect}{{\vec{v}}}
\newcommand{\V}{\field{V}} 
\newcommand{\W}{\field{W}}

\newcommand{\Rn}{\mathbb{R}^{n}}
\newcommand{\hn}{\mathbb{R}^{n-1}}
\newcommand{\R}{\mathbb{R}}

\newcommand{\N}{\mathbb{N}}

\renewcommand{\hn}{\mathbb{G}}

\newcommand{\GG}{\mathbb{G}}

\newcommand{\ha}{\mathcal{H}}
\newcommand{\h}{\mathbb{G}}

\renewcommand{\ss}{\mathcal{S}}

\newcommand{\s}{\sigma}

\newcommand{\stm}{\setminus}

\newcommand{\anah}{\nabla_\GG}
\newcommand{\divh}{\di_\GG}
\newcommand{\g}{\gamma}
\newcommand{\G}{\Gamma}
\renewcommand{\dh}{\mathcal{L}}

\newcommand{\restrict}{\begin{picture}(12,12)
                        \put(2,0){\line(1,0){8}}
                        \put(2,0){\line(0,1){8}}
                       \end{picture}}

\newcommand{\cS}{{\mathcal S}}

\DeclareMathOperator{\rank}{rank}
\DeclareMathOperator{\diver}{div}

\newcommand {\grtrsim} {\ {\raise-.5ex\hbox{$\buildrel>\over\sim$}}\ }
\newcommand{\e}{\varepsilon}

\newcommand{\ra}{\rightarrow}
\newcommand{\khii}{\text{\lower -.4ex\hbox{$\chi$}}}
\DeclareMathOperator{\spt}{spt} \DeclareMathOperator{\dist}{dist}
 
 \DeclareMathOperator{\diam}{diam}
 \DeclareMathOperator{\di}{div}

\DeclareMathOperator{\spa}{span}

\newcommand{\der}{\partial}
\renewcommand{\div}{\mbox{\rm div}}
\newcommand{\lan}{\langle}
\newcommand{\ran}{\rangle}
\renewcommand{\O}{\Omega}

\begin{document}
\title{Removable sets for Lipschitz harmonic functions on Carnot groups}
\author{Vasilis Chousionis}
\address{Department of Mathematics \\ University of Illinois \\ 1409
  West Green St. \\ Urbana, IL, 61801}
\email{vchous@illinois.edu}
\author{Valentino Magnani}
\address{Dipartimento di Matematica, Universit\`a di Pisa, Largo Bruno Pontecorvo 5, 56127, Pisa, Italy}
\email{magnani@dm.unipi.it}

\author{Jeremy T. Tyson}
\address{Department of Mathematics \\ University of Illinois \\ 1409
  West Green St. \\ Urbana, IL, 61801}
\email{tyson@illinois.edu}
\date{\today}
\thanks{VM supported by the European research project AdG ERC
  ``GeMeThNES", grant agreement 246923. JTT supported by NSF grant
  DMS-1201875.}

\begin{abstract}
Let $\GG$ be a Carnot group with homogeneous dimension $Q \geq 3$ and
let $\cL$ be a sub-Laplacian on $\GG$. We prove that the critical
dimension for removable sets of Lipschitz $\cL$-harmonic functions is
$(Q-1)$. Moreover we construct self-similar sets with positive and finite
$\ha^{Q-1}$ measure which are removable.

\end{abstract}

\maketitle
\section{Introduction}
A compact set $K$ in the complex plane is called \textit{removable}
for bounded analytic functions if for any open set $\O$ containing $K$
any bounded analytic function on $\O \stm K$ has an analytic extension
to $\O$. It is easily seen that points are removable while closed
disks are not. Already at the end of the 19th century, Painlev\'e
proved that sets of zero length are removable. He naturally raised
the question of geometrically characterizing removable sets. In 1947
Ahlfors in \cite{Ahl} gave a potential-theoretic characterization of
removable sets by defining the celebrated notion of analytic capacity.
In passing we note that Vitushkin, see e.g.\ \cite{Vi}, used analytic
capacity and a close variant, the so called continuous analytic
capacity, to study problems of uniform rational
approximation on compact sets of the complex plane. Although it
was known by then that the critical dimension for removable sets is
$1$ very few things were known about sets with critical dimension.
The following question arose: {\it is it true that a compact $K$ is
non-removable if and only $\mathcal{H}^1(K)>0$?} Here, $\ha^1$ stands
for the $1$-dimensional Hausdorff measure.  

The negative answer to the above question was obtained by Vitushkin
\cite{Vi2} in the 1960's. Vitushkin constructed a removable compact
set $K$ with $0<\ha^1(K)<\infty$. Later on, Garnett \cite{gar} and
Ivanov \cite{iv} proved that the familiar $1$-dimensional $4$-corners
Cantor set is in fact removable for bounded analytic functions. The
``irregular" geometric structure of these examples led Vitushkin to
conjecture that: \textit{a compact set $K$ is removable if and only if
it is purely unrectifiable}. Recall that a set  $K$ is called
rectifiable if there exist countably many Lipschitz curves $\G_i$
such that $\ha^1(K \stm \cup_i \Gamma_i)=0$.
On the other hand a set is called purely unrectifiable if it
intersects any rectifiable curve in a set of $\ha^1$ measure zero. 
Although Vitushkin's conjecture is false in full generality (this was
proved in an astonishing way by Mattila in \cite{matannals}) it turns
out that it holds if we restrict attention to sets of finite length.
The latter result is due to David \cite{Dav}.  

The proof of Vitushkin's conjecture has a long and interesting history
which is deeply related to the geometric study of singular integrals.
See \cite{Ve}, \cite{M} or \cite{tolsabook} for extensive treatments.
We first remark that the ``if'' part in the restricted conjecture of
Vitushkin follows from Calder\' on's theorem on the $L^2$ boundedness
of the Cauchy transform on Lipschitz graphs with small Lipschitz
constant. It is of interest that Calder\'on studied this problem in
connection with partial differential equations with minimal smoothness
conditions not being aware with the connections to removability. In
subsequent years the topic was studied extensively and several deep
contributions were made, see e.g. \cite{mcm},\cite{pj} and \cite{ds}.
Nevertheless it was Melnikov's discovery in \cite{me} of the relation
of the Cauchy kernel to the so-called Menger curvature that triggered
many advances during the 1990's, which eventually led to the complete
resolution of Vitushkin's conjecture. In \cite{MMV} Mattila, Melnikov
and Verdera proved Vitushkin's conjecture in the particular case where
the set $K$ is 1 Ahlfors--David regular, or in short 1-AD-regular. A
Radon measure $\mu$ is $s$-AD-regular, $s>0$, if 
\begin{equation*}\label{ad}
\frac{r^s}{C} \leq \mu(B(z,r)) \leq C r^s \text{ for }  z \in
\operatorname{spt}\mu \text{ and } 0<r<{\rm
  diam}(\operatorname{spt}(\mu)),
\end{equation*}
for some fixed constant $C$. A set $K$ is $s$-AD regular if the
measure $\ha^s \lfloor K$ is $s$-AD regular. A few years later David
characterized in \cite{Dav} the removable sets of bounded analytic
functions among sets of finite length and Tolsa gave a complete Menger
curvature integral characterization in \cite{To} of all removable
sets of bounded analytic functions. We mention that all these results
depend on the deep geometric study of the Cauchy singular integral.

A compact set $K \subset \R^n$ is said to be removable for Lipschitz
harmonic functions if whenever $D$ is an open set containing $K$ and  $f:D \ra \R$ is a Lipschitz function which is harmonic in $D \stm K$, then $f$ is harmonic in $D$. David and Mattila in \cite{dm}
characterized planar removable sets with finite length: finite length
removable sets for either bounded analytic or Lipschitz harmonic functions
are precisely the purely $1$-unrectifiable sets. 
This is one of the various reasons why Lipschitz harmonic functions are 
a natural class to study. Very recently Nazarov, Tolsa and Volberg 
\cite{ntov2} extended the result
of David and Mattila in $\Rn$ by proving that a compact set $K \subset
\Rn$ with $\ha^{n-1}(K)<\infty$ is removable for Lipschitz harmonic
functions if and only if it is purely $(n-1)$-unrectifiable. We should mention here
that both results depend heavily on singular integrals. The result of
David and Mattila is based on intricate $Tb$ theorems for
non-doubling measures and the Cauchy transform. Nazarov, Tolsa and 
Volberg base their proof on their earlier very deep work \cite{ntov},
where they prove that if $\mu$ is an $(n-1)$-AD regular measure, then
the Riesz kernel $x/|x|^n, x\in \Rn \stm \{0\},$ defines bounded
singular operators in $L^2(\mu)$ if and only if $\mu$ is $(n-1)$-uniformly
rectifiable. Uniform rectifiability can be thought as a quantitative
version of rectifiability. The Riesz kernels arise naturally in the
study of removable sets for Lipschitz harmonic functions, as one
readily sees that $\nabla \G_n=x/|x|^{n}, x \in \Rn \setminus \{0\}$,
where $\G_n=c_n |x|^{2-n}$ denotes the fundamental solution of the
Laplacian for $n \geq 3$.

Recently, significant effort has been made towards the extension of
classical Euclidean analysis and geometry into general non-Riemannian
spaces, including Carnot groups and more abstract metric measure
spaces. In particular, potential theory related to sub-Laplacians in
Carnot groups is an active research field with many recent
developments, see \cite{blu} and the references given there. In
\cite{chousionis-mattila} the problem of removability for Lipschitz
$\cL$-harmonic functions in the Heisenberg group $\mathbb{H}^n$ was
considered. It was established there that, in accordance with the
Euclidean case, the critical removability dimension is $Q-1$, where
$Q=2n+2$ denotes the Hausdorff dimension of the Heisenberg group.
Moreover, examples of separated self-similar removable sets with
positive and finite $(Q-1)$-measure were given. An essential
ingredient in order to establish the existence of such sets was the
proof of a general criterion for unboundedness of singular integrals
on self similar sets of metric groups.

The aim of the present paper is to extend the results from
\cite{chousionis-mattila} to general Carnot groups. Our first result
reads as follows.
%In order to to be consistent with the convention in
%\cite{chousionis-mattila} and \cite{ch-tys}, we employ the following
%definition of removability: For the definition of the sub-Laplacian $\dh$, see section\ref{sec:background}.
\begin{thm}\label{pos}
Let $C$ be a compact subset of a Carnot group $\hn$ and denote by $Q$ the homogeneous
dimension of $\hn$. Let $\cL$ be the sub-Laplacian in $\GG$.
\begin{enumerate} 
\item If $\mathcal{H}^{Q-1}(C)=0$, $C$ is removable for Lipschitz
  $\cL$-harmonic functions.
\item If $\dim C>Q-1$, $C$ is not removable for Lipschitz
  $\cL$-harmonic functions.
\end{enumerate}
\end{thm}

The proof of Theorem \ref{pos} is similar to the proof from
\cite{chousionis-mattila}. Nevertheless we decided for the convenience
of the reader to provide all of the details, although in some places the
arguments are identical to those in \cite{chousionis-mattila}. As in
\cite{chousionis-mattila} the proof of Theorem \ref{pos} relies on a
representation theorem for Lipschitz $\cL$-harmonic functions (Theorem
\ref{main}). The analogue of Theorem~\ref{main} in
\cite{chousionis-mattila} uses the divergence theorem of Franchi,
Serapioni and Serra-Cassano \cite{fss} which is known to be true only
for step two Carnot groups. 

In the case of general Carnot groups, we overcome this obstacle using the 
Euclidean regularity of the domains appearing in the proof of Theorem~\ref{main}. 
In fact, we have finite unions of bounded sets with smooth boundary, that are
sets of finite perimeter in the Euclidean sense.
On the other hand, we have also to detect the Euclidean reduced boundary, which 
we accomplish by perturbing a given piecewise smooth boundary and using the classical Sard's theorem.
Then joining the Euclidean divergence theorem for finite perimeter sets, \cite{DeGiorgi}, 
with area-type formulae for the sub-Riemannian spherical Hausdorff measure of smooth sets,
\cite{valentino1}, \cite{valentino2}, we reach the sub-Riemannian divergence formula
in this special class of domains.

Additional technical difficulties arise from the fact that, while the
fundamental solution $\G$ of the sub-Laplacian in the Heisenberg group
has an explicit formula, the corresponding fundamental solution for
general sub-Laplacians in general Carnot groups admits no such
formula. Nevertheless the fundamental solution is always
$(2-Q)$-homogeneous and this fact is essential in our proofs.

We also study the critical case (dimension $Q-1$). It is easy to
construct nonremovable sets of positive and finite $\cH^{Q-1}$ measure
(see Remark \ref{verplane}). Our second main theorem reads as follows.

\begin{thm}\label{unbrem} 
There exist sets  $K \subset \GG$ with $0<\mathcal{H}^{Q-1}(K)<
\infty$ which are removable for Lipschitz $\cL$-harmonic functions.
\end{thm}

In \cite{chousionis-mattila} such sets were constructed in the
Heisenberg group $\mathbb{H}^n$ based on Strichartz-type tilings, see
\cite{str:self-sim1}. However in general Carnot groups such tilings do
not exist, and we provide an alternate constructive argument involving
separated self-similar Cantor subsets in vertical subgroups of $\GG$.
As in the Euclidean case we need to consider singular integrals with
respect to the kernel $k= \anah \G$, which is $(1-Q)$-homogeneous.
Roughly speaking, if one is able to prove that a certain singular
integral is unbounded on $L^2(\cH^{Q-1} \lfloor K)$, then the set $K$ is
removable. Our idea is to construct a separated self similar set $K$,
with $0<\ha^{Q-1}(K)<\infty$, which lives on a dilation cone where at
least one coordinate of the kernel $k$ keeps constant sign. Moreover
the set $K$ is constructed in such a way that it has a fixed point at
the origin. These properties enable us to apply directly the
unboundedness criterion for singular integrals on self similar sets 
from \cite{chousionis-mattila} (reproduced in this paper as Theorem
\ref{unb}).

Removability of sets can be studied for other partial differential
equations, and in other regularity classes. In \cite{ch-tys},
quantitative estimates on the size of removable sets for solutions of
a wide variety of partial differential equations in Carnot groups are
given.

The paper is organised as follows. In section \ref{sec:background} we
lay down the necessary background in Carnot groups as well as some
basic properties of their sub-Laplacians. In section \ref{sec:crit} we
prove a representation theorem for Lipschitz $\cL$-harmonic functions
outside some compact set $K$, namely Theorem \ref{main}, and this
leads to the proof of Theorem \ref{pos}. In section \ref{sec:rem} we
provide examples of removable sets with positive and finite
$\ha^{Q-1}$-measure.

\section{Definitions and notation}\label{sec:background}

A \define{Carnot group} is a connected, simply connected and nilpotent Lie group $\GG$,
with graded Lie algebra
$$
\fg=\fv_1\oplus\cdots\oplus \fv_s\,,
$$
such that $[\fv_1,\fv_i]=\fv_{i+1}$ for $i=1,2,\ldots,s-1$ and
$[\fv_1,\fv_s]=0$.
Under these conditions the exponential mapping $\exp:\fg\to\GG$ is bianalytic, hence 
we can canonically identify elements $\fg$, namely left invariant vector fields, with elements of $\GG$.
The integer $s \ge 1$ is the \define{step} of $\GG$. 
We denote the group law in $\GG$ by $\cdot$ and the identity element of $\GG$ by $0$.

We fix an inner product $\langle \ , \ \rangle$ in $\fv_1$ and let $X_1,\ldots,X_{m}$ be an orthonormal basis for $\fv_1$ relative to this inner product. Using
this basis, we construct the \define{horizontal subbundle} $H\GG$ of
the tangent bundle $T\GG$ with fibers $H_p\GG=\spa\lbrace X_1(p),
\ldots, X_{m}(p)\rbrace$, $p\in \GG$. A left-invariant vector field $X$
on $\GG$ is \define{horizontal} if it is a section of $H\GG$. The inner
product on $\fv_1$ defines a left invariant family of inner products
on the fibers of the horizontal subbundle.

We denote by $d$ the \define{Carnot--Carath\'eodory metric} on $\GG$,
defined by infimizing the lengths of horizontal paths joining two
fixed points, where the horizontal length is computed using the
aforementioned inner product. More specifically we define:

\begin{df}
\label{subunit}
An absolutely continuous curve $\gamma:[0,T]\ra \hn$ will be called sub-unit, with respect to the vector fields $X_1,\dots,X_m,$ if there exist real measurable functions $a_j:[0,T]\to\R$, with $j=1,\ldots,m$,
such that $\sum_{j=1}^{m}  a_j(t)^2 \leq 1$ for a.e.\ $t\in[0,T]$ and 
$$\dot \g (t)=\sum_{j=1}^m a_j(t)X_j(\g (t))\ \ \text{for a.e.}\ t\in[0,T].$$
\end{df}

\begin{df}
\label{cc}
For $p,q \in \hn$ their Carnot-Carath\'eodory distance is
\begin{equation*}
\begin{split}
d(p,q)=\inf \{T>0:\ &\text{there is a sub-unit curve} \ \g:[0,T]\ra \hn \\ 
&\quad \quad \quad \text{such that} \ \g(0)=p \ \text{and} \ \g(T)=q\}.
\end{split}
\end{equation*}
\end{df}
It follows by Chow's theorem that the above set of curves joining $p$ and $q$ is not empty and hence $d$ is a metric on $\hn$. The closed and open balls with respect to $d$ will be denoted by $B(p,r)$ and $U(p,r)$ respectively.

For each $t>0$, we define $\delta_t:\fg\to\fg$ by setting $\delta_t(X)=t^iX$ if $X\in \fv_i$ and extending 
the mapping by linearity. The identification of the Lie algebra with the Lie group via the exponential mapping allows us to introduce dilations on $\GG$, that we also denote by $\delta_t$. 
Then $(\delta_t)_{t>0}$ is the
one-parameter family of \define{dilations} of $\GG$ satisfying
$d(\delta_t(p),\delta_t(q))=t d(p,q)$ for $p,q\in \GG$. Another family
of automorphisms in $\GG$ are the \define{left translations}
$\tau_q:\GG \ra \GG$ defined by $\tau_q(x)= q\cdot x,x \in \GG,$ 
for all $q \in \GG.$ We note also that the metric $d$ is left
invariant, i.e., $d( q \cdot p_1)=d(q \cdot p_2)$ for $q,p_1,p_2 \in
\hn$.

The Jacobian determinant of $\delta_t$ (with respect to Haar measure) is
everywhere equal to $t^Q$, where
$$
Q=\sum_{i=1}^s i\dim \fv_i
$$
is the \define{homogeneous dimension} of $\GG$.
{\bf In this paper, we always assume $Q\ge 3$.}

A measurable function $f$ on $\GG$ will be called
$\lambda$-\define{homogeneous}, or homogeneous of degree $\lambda$, if
$f \circ \delta_t=t^\lambda f$ for all $t>0$. A continuous function
$\|\cdot\|:\GG \ra [0,\infty)$ is called a \define{homogeneous norm} if  
 $\|\delta_t(p)\|=t\|p\|$ for all $t>0$ and $p \in \GG$ and $\|p\|>0$
 for all $p \neq 0$. A typical example of a homogeneous norm is the
 function
$$
\|p\|_{cc}:=d(p,0).
$$
All homogeneous norms in $\GG$
are equivalent: recall that two norms $\|\cdot\|_1$ and $\|\cdot\|_2$
are said to be equivalent if there exists a positive constant $c$ such
that
\begin{equation}\label{companorms}
c^{-1}\|p\|_2 \leq \|p\|_1 \leq c \|p\|_2 \qquad \mbox{for all $p \in \GG$.}
\end{equation} 
Proofs of these facts, as well as other properties of homogeneous
norms, can be found in~\cite{blu}.

Since $\GG$ is identified with the linear space $\fg$, we can fix a {\em graded basis} of $\fg$,
hence we can identify elements of $\GG$ with elements of $\R^N$,
where $N=\sum_{i=1}^s\dim \fv_i$.  A graded basis in $\fg$ respects the grading, that is there
exists $s$ ordered subsets of the basis that are in turn bases of the single layers $\fv_i$. 
One can check that translations with respect to graded coordinates preserve the Lebesgue measure
in $\R^N$. As a consequence, the Haar measure on $\GG$ can be obtained from the Lebesgue measure on $\R^N$.
It also agrees (up to a constant) with the $Q$-dimensional Hausdorff measure in the metric space $(\GG,d)$. 

In this paper we will denote the Haar measure of a set $E \subset \GG$ by
$|E|$, and we will write integrals with respect to this measure as
$\int_E f(x) \, dx$ or $\int_E f$. We refer the reader to
\cite{mont:tour}, \cite{blu} or \cite{cdpt:survey} for further
information on Carnot groups and their metric geometry.

In particular a fixed basis $X_1,\ldots,X_m$ of the first layer $\fv_1$ is fixed.
This is the so-called horizontal frame, that linearly spans all of the horizontal directions.
If $f$ is a real function defined on an open set of $\hn$ its
$\h$-gradient is given by 
$$\nabla_\h f=(X_1f,\dots,X_m f).$$
The $\h$-divergence of a function $\phi=(\phi_1,\dots,\phi_{m}):\hn
\ra \R^{m}$ is defined as 
$$\di_\h \phi=\sum_{i=1}^m X_i\phi_i.$$

\begin{rem}\label{valerem}
For our purposes, a sub-Riemannian divergence theorem is necessary.
We will deal with regular domains comprised of finite unions of smooth
open and bounded sets. 
%In particular, finite unions of finite perimeter sets are still finite
%perimeter sets. 
Let $\div$ denote the standard divergence in $\R^N$ and let $X$ be a
$C^1$ smooth vector field on $\R^N$ with 
\[
 X=(a_1,\ldots,a_N)\sim a_1\der_{x_1}+\cdots a_N\der_{x_N}.
\]
If $\Omega$ is a bounded set of finite perimeter and $f$ is a $C^1$
smooth real valued function on an open neighborhood of
$\overline\Omega$, then
\begin{equation}\label{div}
 \int_\Omega Xf =\int_{\cF^*\Omega} f\ \lan X,\nu \ran \;
 d\|\partial\Omega\|-\int_\Omega f\div X, 
\end{equation}
where $\lan\cdot,\cdot\ran$ denotes the Euclidean scalar product,
$\nu$ is the generalized outer normal to $\Omega$, $\cF^*\Omega$ is
the reduced boundary and $\|\partial\Omega\|$ is the perimeter measure
of $\Omega$, \cite{DeGiorgi}. The validity of \eqref{div} is seen from
the following equalities:
\[
\begin{split}
 \int_\Omega Xf &= \int_\Omega \sum_{l=1}^N a_l\, \der_{x_l}f
=\int_\Omega \sum_{l=1}^N \Big(\der_{x_l}(a_l f)-f\der_{x_l}a_l\Big)  \\
&=\int_\Omega\div(f X)-\int_\Omega f\div X \\
&=\int_{\cF^*\Omega} f\,\lan X,\nu\ran\,d\|\der\Omega\|-\int_\Omega f\div X .
\end{split}
\]

All the left invariant vector fields $X$ of a Carnot group satisfy
$\div X=0$. As a corollary of \eqref{div}, it follows therefore that
\begin{equation}\label{valentino}
 \int_\Omega \divh F=\int_{\cF^*\Omega} \sum_{j=1}^m f_j\,\lan X_j,\nu\ran\, d\|\der\Omega\|=
\int_{\cF^*\Omega} \lan F,{\bf \nu_\GG}\ran\,d\|\der\Omega\|\,,
\end{equation}
where ${\bf \nu_\GG}=\big(\lan X_1,\nu\ran,\ldots,\lan
X_m,\nu\ran\big)$ is the {\em non-normalized horizontal normal}.
\end{rem}

The sub-Laplacian in $\hn$ is given by
$$\dh= \sum_{i=1}^m X_i^2$$
or equivalently 
$$\dh=\divh \nabla_\h .$$ 

\begin{df}\label{harmdef}
Let $D\subset \hn$ be an open set. A real valued function $f \in
C^2(D)$ is called $\dh$-harmonic, or simply harmonic, on $D$ if $\dh
f= 0$ on $D$.
\end{df}
%^Actually, the assumption $f \in C^2(D)$ is superfluous, since even the distributional solutions of $\dh f= 0$ are $C^{\infty}$, see \cite{blu}.

We shall consider removable sets for Lipschitz solutions of the sub-Laplacian:

\begin{df}
\label{rem} 
A compact set $C \subset \hn$ will be called removable, or $\dh$-removable for Lipschitz $\dh$-harmonic functions, if for every domain $D$ with $C \subset D$ and every Lipschitz function $f:D \ra \R$,
$$\dh f=0 \ \text{in} \  D\stm C \ \text{implies} \ \dh f = 0 \ \text{in} \ D.$$
%$$\anah f|_{D\stm C}=0\ \text{implies} \ \anah f|_{D}=0$$
\end{df}

As usual we denote for any $D \subset \hn$ and any function $f:D \ra \R$,
$$\text{Lip} (f):= \sup_{x,y \in D} \frac{|f(x)-f(y)|}{d(x,y)},$$
and we will also use the following notation for the upper bound for the Lipschitz constants in  Carnot-Carath\'eodory balls:
$$\text{Lip}_{\text{B}} (f):= \sup \{\text{Lip} (f|_{U_c(p,r)}):p \in D, r>0, U_c(p,r) \subset D\}.$$

The following proposition is known. It follows, for example, from the Poincar\'e inequality, see Theorem 5.16 in \cite{cdpt:survey} and the arguments for its proof on pages 106-107. A simple direct proof which applies directly in our setting can be found in \cite{chousionis-mattila}.

\begin{pr}
\label{panulip}
Let $D \subset \hn$ be a domain and let $f \in C^1(D)$. Then $\text{Lip}_{\text{B}} (f) < \infty$ if and only if $\|\anah f\|_\infty < \infty$. More precisely, there is a constant $c(\GG)$ depending only on $\GG$ such that
\begin{equation}\label{panulip1}
\|\anah f\|_\infty\leq \text{Lip}_{\text{B}} (f)\leq c(\GG)\|\anah f\|_\infty.
\end{equation}
\end{pr}

Fundamental solutions for sub-Laplacians in homogeneous Carnot groups are defined in accordance with the classical Euclidean setting.
\begin{df}[Fundamental solutions] A function $\Gamma : \R^{N} \stm
  \{0\} \ra \R$ is a fundamental solution for $\dh$ if: 
\begin{enumerate}
\item $\G \in C^{\infty}(\R^{N} \stm \{0\})$,
\item $\G \in L_{\text{loc}}^1(\R^{N})$ and $\lim_{\|p\|_{cc} \ra \infty}\G(p)\ra 0$,
\item for all $\varphi \in C_0^{\infty} (\R^{N})$,
$$\int_{\R^{N}}\G(p) \dh \varphi (p) \, dp=-\varphi (0).$$
\end{enumerate}
\end{df}

It also follows easily, see Theorem 5.3.3 and Proposition 5.3.11 of
\cite{blu}, that for every $p \in \hn$, 
\begin{equation}
\label{fundconv}
\G \ast \dh \varphi (p)=-\varphi(p) \ \text{for all} \ \varphi \in C_0^{\infty} (\R^{N}).
\end{equation}
Convolutions are defined as usual by $$f\ast g (p)=\int f(q^{-1}\cdot p)g(q)\,dq$$
for $f,g \in L^1$ and $p \in \hn$.

A very general result due to Folland \cite{fol} guarantees the
existence of a fundamental solution for each sub-Laplacian on a
homogeneous Carnot group with homogeneous dimension $Q\ge 3$. The
following proposition gathers some well-known properties of such
fundamental solutions. Proofs can be found in \cite{blu}.  

\begin{pr}[Properties of $\G$] Let $\G$ be the fundamental solution of
  $\dh$. Then for all $p \in \GG \stm \{0\}$ and all $t>0$:
\begin{enumerate}
\item (\textit{Symmetry}) $\G(p^{-1})=\G(p)$,
\item ($\delta_t$-\textit{homogeneity}) $\G(\delta_t(p))=t^{2-Q}\G(p)$,
\item (\textit{Positivity}) $\G(p)>0$.
\end{enumerate}
\end{pr}

The function
$$
\|p\|_\G = \begin{cases} \G(p)^{\frac{1}{2-Q}} &\mbox{if } p \in \GG
  \stm \{0\} \\ 0 & \mbox{if } p=0.  \end{cases}
$$
is a symmetric homogeneous norm which is $C^\infty$ away from the
origin. Let 
$$
d_\G (p,q)=\|p^{-1} \cdot q\|_\G
$$
be the quasi-distance defined by $\|\cdot\|_\G$. We will denote the
corresponding open and closed balls by $U_\G (p,r)$ and $B_\G (p,r)$
respectively. Note also that by (\ref{companorms}) $d$ and $d_\G$ are
globally equivalent. 

Let $k=\anah \Gamma$, then $k=(k_1,\dots,k_{m}):\GG \stm \{0\} \ra
\R^{m}$, and  
$$
k(p)=\anah \G (p)=\anah (\|p\|_\G^{2-Q})=(2-Q)\frac{\anah
  \|p\|_\G}{\|p\|_\G^{Q-1}} := \frac{\Omega (p)}{\|p\|_\G^{Q-1}} 
$$
for $p \in \GG \stm \{0\}$. Furthermore $\Omega$ is smooth in $\GG
\stm \{0\}$  and $\delta_t$-homogeneous of degree zero, which in
particular implies that $k$ is $(1-Q)$-homogeneous and  
\begin{equation}
\label{sizeest}
|k(p)| \lesssim \|p\|_\G^{1-Q}
\end{equation}
for $p \in \GG \stm \{0\}$. Notice also that
\begin{equation}
\label{coorker}
k_i (p)=\frac{\O_i(p)}{\|p\|_\G^{Q-1}}, \qquad p \in \GG \stm \{0\},
\end{equation}
where $\O=(\O_1, \dots, \O_m)$ and every function $\O_i$ is smooth and
homogeneous of degree zero.

We denote by $\mathcal{H}^s,s\geq 0,$ the $s$-dimensional Hausdorff
measure obtained from the Carnot-Caratheodory metric $d$, i.e. for $E
\subset \GG$ and $\delta >0$, $\mathcal{H}^s (E)=\sup_{\delta>0}
\mathcal{H}^s_\delta (E)$, where 
$$
\mathcal{H}^s_\delta(E)=\inf \left\{\sum_i  \diam(E_i)^s: E \subset
\bigcup_i E_i,\diam (E_i)<\delta \right\}.
$$
In the same manner the $s$-dimensional spherical Hausdorff measure for
$E \subset \GG$ is defined as $\ss^s (E)=\sup_{\delta>0} \ss^s_\delta
(E)$, where 
$$
\ss^s_\delta(E)=\inf \left\{\sum_i  r^s_i: E \subset \bigcup_i
  B(p_i,r_i),r_i \leq \delta,p_i \in \GG \right\}.
$$
We will denote by $\cH^s_\G$ and $\cS^s_\G$ the Hausdorff and
spherical Hausdorff measures with respect to  $d_\G$. Since
homogeneous norms are equivalent it follows that the measures $\cH^s,
\cS^s$, $\cH^s_\G$ and $\cS^s_\G$ are all mutually absolutely
continuous with bounded Radon-Nikodym derivatives. 

\section{The critical dimension for $\dh$-removable sets}\label{sec:crit}

We first prove a representation theorem for Lipschitz harmonic
functions outside compact sets of finite $\ha^{Q-1}$ measure.

\begin{thm}\label{main}
Let $C$ be a compact subset of $\hn$ with  $\mathcal{H} ^{Q-1} (C) <
\infty$ and let $D\supset C$ be a domain in $\hn$. Suppose $f:D \ra
\R$ is a Lipschitz function such that  $\dh f = 0$ in $D \stm C$. 
Then there exist a bounded domain $G$, $C \subset G \subset D$, a
Borel function $h:C\ra \R$ and an $\dh$-harmonic function $H:G \ra \R$
such that
$$
f(p)=\int_C \G(q^{-1}\cdot p)h(q) \, d \ha^{Q-1}(q) + H(p) \ \text{for} \ p
\in G \stm C
$$
and $\|h\|_{L^\infty(\ha^{Q-1} \lfloor C)} +\|\anah H \|_\infty  \lesssim 1$.
\end{thm}

\begin{proof}
Let $D_1$ be a domain such that $C \subset D_1 \subset D$, $\bar{D_1}$
is compact and $\dist (\overline{D_1},\hn \stm D)>0$. For every
$m=1,2,\dots$ there exists a finite number of balls
$U_{m,j}:=U_\G(p_{m,j},r_{m,j})$, $j=1,\ldots,j_m$, such that
$U_{m,j}\cap C \neq \emptyset$,
\begin{equation}\label{gmmeas}
C \subset \bigcup_{j=1}^{j_m} U_{m,j} \subset D_1, \ \ r_{m,j} \leq
\frac{1}{m},
\end{equation}
and
\begin{equation}\label{gmmeas2}
\sum_{j=1}^{j_m} r_{m,j}^{Q-1} \leq \ss_\G ^{Q-1} (C)+ \frac{1}{m}.
\end{equation}

Temporarily fix $m \in \N$, and for simplicity let $p_j:=p_{m,j}$ and
$r_j:=r_{m,j}$. The boundary of the union of the balls, $\bigcup_j
U_\G(p_j,r_j)$, is contained in the union of the boundaries, and hence
has (Euclidean) dimension at most $N-1$. We want to show that the
overlap set
$$
\bigcup_{j\ne i} \partial U_\G(p_i,r_i) \cap \partial U_\G(p_j,r_j)
$$
is a null set for the Euclidean Hausdorff $(N-1)$-measure, in order to ensure that it is negligible for the classical divergence theorem. This follows from Sard's theorem, provided we adjust the radii slightly.

Since $C$ is compact and the balls $U_\G(p_j,r_j)$ are open, we have room to decrease the radii slightly while still covering $C$.

\begin{lm}\label{overlap-lemma}
Assume the centers $p_j$ are distinct, and fix intervals $J_j = [r_j-\epsilon,r_j]$ for some $\epsilon>0$. Then there exist values $r_j' \in J_j$ so that
$$
\dim_E \left( \bigcup_{i\ne j} \partial U_\G(p_i,r_i') \cap \partial U_\G(p_j,r_j') \right) \le N-2.
$$
Consequently, $\cup_{j\ne i} \partial U_\G(p_i,r_i') \cap \partial U_\G(p_i,r_i')$ is a null set for the measure $\cH^{N-1}_E$.
\end{lm}

Here $\dim_E$ refers to the dimension in the underlying Euclidean metric of $\R^N$.

\begin{proof}
It suffices to assume $j_m=2$. We wish to show that
$$
\dim_E ( \partial U_\G(p_1,r_1') \cap \partial U_\G(p_2,r_2') ) \le N-2
$$
for some $r_1' \in J_1$, $r_2' \in J_2$. Consider the map $F:\GG \to
\R^2$ given by
$$
F(p) = \left( d_\G(p,p_1) , d_\G(p,p_2) \right).
$$
Then $F$ is $C^\infty$, and $F^{-1}
(r_1',r_2')
 = \partial U_\G(p_1,r_1') \cap \partial U_\G(p_2,r_2')$ for $r_1',r_2'>0$. According to Sard's theorem \cite{sard}, the set of critical values of $F$ has measure zero in $\R^2$. Since $J_1\times J_2$ has positive measure, there exist $r_1' \in J_1$, $r_2' \in J_2$ so that $(r_1',r_2')$ is a regular value of $F$, i.e., $\rank DF(p) = 2$ for all $p \in F^{-1}(r_1',r_2')$. Moreover, the set $F^{-1}(r_1',r_2')$ is a smooth submanifold whose (Euclidean) dimension is at most $\dim_E \GG - \dim_E \R^2 = N-2$.
\end{proof}

The balls $U_\G(p_j,r_j')$ continue to cover $C$ and satisfy
\eqref{gmmeas} and \eqref{gmmeas2}. In view of the above, we can assume without loss of generality that the conclusion of the lemma holds for the original balls $U_\G(p_j,r_j)$ (i.e., we relabel $r_j'$ as $r_j$).

The Dimension Comparison Theorem in Carnot groups (see Theorem 2.4 and
Proposition 3.1 in \cite{btw}), in codimension one, implies that the
spherical Hausdorff measure $\cS^{Q-1}_\Gamma$ constructed from the
metric $d_\Gamma$ for a fixed homogeneous distance $\Gamma$ is bounded
above (up to a constant) by the Euclidean measure $\cH^{N-1}_E$. It
follows from this and Lemma~\ref{overlap-lemma} that the overlap set
is also a null set for the spherical Hausdorff measure $\cS^{Q-1}_\G$. 
%This measure is the
%$(Q-1)$-dimensional spherical Hausdorff measure constructed from the
%homogeneous distance $d_\G$.

Let $G_m = \cup_{j=1}^{j_m} U_{m,j}$ and $$0<\e_m <\min\{1,\dist(C,\hn
\stm G_m),\dist(G_m,\hn \stm D_1) \}.$$ 
By the Whitney-McShane Extension Lemma there exists a Lipschitz
function $F:\hn \ra \R$ such that $F|_{D}=f$ and $F$ is bounded.

%Let $$D_{\e_m}:=\{p \in D_1 \stm C :  \dist(p, \partial (D_1 \stm C))> \e_m\},$$
%and
If $d_0=1+\max_{z\in\overline{D_1}}d(z,0)$, then the condition $d(y,0)+d(z,0) \leq d_0$ gives
\begin{equation}
\label{homodist}
d(y^{-1} \cdot z, z) \leq c(d_0)d(y,0)^{1/s},
\end{equation}
due to \cite[3.18]{valentinoaus}.
Let $\Phi \in C^\infty_0 (\R^{N})$, $\Phi \geq 0$, such that $\spt
\Phi \subset U(0,1)$ and $\int \Phi =1$. For any $\delta>0$ let
$\Phi_\delta(x)=\delta^{-Q} \Phi(\delta_{1/\delta}(x))$. We consider the sequence
of mollifiers
\begin{equation}\label{mol}
\begin{split}
f_m(x):=F \ast \Phi_{\delta_m}(x)=\int F(y) \Phi_{\delta_m}(x \cdot y^{-1})\,dy
=\int_{U(0,\delta_m)} F(y^{-1}\cdot x) \Phi_{\delta_m}(y)\,dy
\end{split}
\end{equation}
for $x \in \hn$ and $\delta_m=(\frac{\e_m}{2c(d_0)})^s$. Since $F$ is bounded and uniformly continuous,
$$
\|f_m-F\|_\infty \ra 0
$$ 
on compact sets of $\GG$. Furthermore for all $m \in \N$, we have that
\begin{enumerate}
\item $f_m \in C^\infty$,
\item $\|\anah f_m\|_\infty \leq \|\anah F\|_\infty < \infty$.
\end{enumerate}
For $\delta>0$ and $S \subset \GG$ let
\begin{equation*}
\label{neig}L(S,\delta) = \{ p \in S \, : \, \dist(p,S^c) > \delta \}.
\end{equation*} 
%For all $\delta>0$, we define 
%\begin{equation}
%\label{neig}L(D_1\stm C,\delta) = 
%\{ p \in \hn \, : \, \dist(p,C) > \delta\ \mbox{and}\ \dist(D_1^c,p)>\delta\}.
%\end{equation}
If $x \in L(D_1 \stm C, \e_m)$, $y \in B(0,\delta_m)$ and $z \in C$, by (\ref{homodist}) we obtain
\begin{equation*}
\begin{split}
d(y^{-1} \cdot x, z)&\geq d(x,z)-d(y^{-1}\cdot x, x)>\e_m -c(d_0)d(y,0)^{1/s}>0.
\end{split}
\end{equation*}
In particular $y^{-1} \cdot x \notin C$ and in the same way $y^{-1} \cdot x \notin \GG \setminus D_1$. Therefore every mollifier $f_m$ is harmonic in  $D_{\e_m}$. We continue by choosing another domain $D_2$ such
that $G_m \subset D_2 \subset L (D_1,\e_m )$ for all $m=1,2,\dots$, and an
auxiliary function $\varphi \in C_0^\infty(\R^{N})$ such that 
$$ 
\varphi  = \begin{cases}
       1 & \text{in} \ D_2\\
       0 & \text{in} \ \hn \stm \overline{D}_1.\\
     \end{cases}
$$
For $m=1,2,\dots$ set $g_m:= \varphi f_m$ and notice that $g_m \in
C_0^\infty(\R^{N})$ and  
$$\|\anah g_m\|_\infty \leq A_1$$
where $A_1$ does not depend on $m$. It follows by (\ref{fundconv})
that for all $m \in \N$,
\begin{equation}\label{convgm}
-g_m(p)=\G \ast \dh g_m (p)\ \text{for all} \  p \in \hn.
\end{equation}
Notice that
\begin{enumerate}
\item $g_m=0$ in $\hn \stm \overline{D_1}$,
\item $g_m=f_m$ in $D_2 \stm G_m$ and hence $\dh g_m=\dh f_m=0$ in $D_2 \stm G_m$.
\end{enumerate}
Therefore for all $m \in \N$ and $p \in D_2 \stm G_m$,
\begin{equation}\label{finconv}
-f_m (p)=\int_{G_m} \G (q^{-1} \cdot p) \dh g_m (q) \, dq +
\int_{\overline{D_1} \stm D_2} \G(q^{-1} \cdot p)  \dh g_m (q) \, dq 
\end{equation}
by (\ref{convgm}). For $m \in \N$ set $H_m: D_2 \ra \R$ to be
\begin{equation}\label{hfun}
H_m(p)= -\int_{\overline{D_1} \stm D_2} \G(q^{-1} \cdot p)  \dh g_m
(q) \, dq
\end{equation}
and $I_m:D_2 \stm G_m \ra \R,\ m=1,2,\dots$ to be
\begin{equation}\label{imfun}
I_m (p)= -\int_{G_m} \G (q^{-1} \cdot p) \dh g_m (q) \, dq.
\end{equation}
Since the functions $\dh g_m$ are uniformly bounded in $\overline{D_1}
\stm D_2$, for all $m \in \N$   
\begin{enumerate}
\item $H_m$ is harmonic in $D_2$,
\item $\|\anah H_m\|_\infty \lesssim 1$, since $\anah \G$ is locally integrable.
\end{enumerate}
The functions $H_m$ are $C^{\infty}$ by H\"ormander's theorem, see for
example Theorem 1 in Preface of \cite{blu}. Thus we can apply
Proposition \ref{panulip} and conclude from (ii) that
$\text{Lip}_{\text{B}} (H_m) \lesssim 1$.

The functions $I_m$ can be expressed as
\begin{equation}\label{im1}
I_m (p)=-\int_{G_m} \di_{\GG,q} (\G (q^{-1} \cdot p) \anah g_m
(q)) \, dq + \int_{G_m}\langle \anah \G (p^{-1} \cdot q), \anah g_m
(q) \rangle \, dq, 
\end{equation}
where $\di_{\GG,q}$ stands for the $\GG$-divergence with respect to
the variable $q$ and we also used the left invariance of $\anah$ and
the symmetry of $\G$ to get that $$\nabla_{\GG,q}(\G(q^{-1} \cdot
p))=\nabla_{\GG,q}(\G(p^{-1} \cdot q))=\anah \G (p^{-1} \cdot q).$$ 

By (\ref{valentino}) one has the identity
\begin{equation*}
\int_\Omega \diver_\GG F = \int_{\cF^*\Omega} \langle F,\nu_\GG\rangle \, d\|\der\Omega\|
\end{equation*}
for every $C^1$ horizontal vector field $F$ and bounded $C^1$ smooth domain
$\Omega$, where $\nu_\GG$ denotes the non-normalized horizontal normal
introduced in Remark~\ref{valerem}. For instance, we may take $\Omega
= U_\G(p_{j,m},r_{j,m})$ as above for each $m$ and $j$. In fact, in
this case Lemma \ref{overlap-lemma} implies that the overlap of the
boundaries is a null set for the $\cH^{N-1}_E$ measure, whence  
\[
\cF^*\Omega=\der\Omega\stm N\,,
\]
where $\cH^{N-1}_E(N)=0$ and the generalized outer normal $\nu$
coincides with the classical outer normal of $\Omega$ at smooth points
of $\der\Omega$. Since the restriction of the perimeter measure to the
reduced boundary is the $(N-1)$-dimensional Hausdorff measure, it
follows that
\begin{equation}\label{valentino-identity}
\int_\Omega \diver_\GG F = \int_{\der\Omega} \langle F,\nu_\GG\rangle
\, d\cH^{N-1}_E\,.
\end{equation}
Next we want to show that the identity
\begin{equation}\label{identity-of-measures}
 |\nu_\GG|\, \cH^{N-1}_E\restrict \partial\Omega=\alpha\,\cS^{Q-1}
 \restrict \partial\Omega 
\end{equation}
holds for such piecewise $C^1$ domains $\Omega$, where $\alpha$ is a
Borel function on $\der\Omega$ and $\cS^{Q-1}_\Gamma$ is the spherical
Hausdorff measure with respect to a fixed homogeneous distance
$d_\Gamma$. Since $\der\Omega$ is not $C^1$, we cannot apply directly
the area formula of \cite{valentino2} to represent the spherical
Hausdorff measure, as that formula is restricted to $C^1$ %smooth
domains and arbitrary auxiliary Riemannian metrics.

%obtain the measure $|\nu_\GG|\,\cH^{N-1}_E\restrict\der\Omega$, 
%In fact, we can take this Riemannian metric to be equal to the
%Euclidean metric.

We proceed as follows. The 
%restriction of \eqref{identity-of-measures} to the 
overlap of the boundaries is $\cH_E^{N-1}$ negligible, hence it also
$\cS^{Q-1}_\Gamma$ negligible due to Proposition~3.1 of \cite{btw}.
The restriction of \eqref{identity-of-measures} to the smooth parts of
$\der\Omega$ where $\nu_\GG$ vanishes easily follows from the
$Q-1$-dimensional negligibility of characteristic points, see 
\cite{valentino}. In fact, these points are characterized by the
vanishing of the horizontal normal $\nu_\GG$. Finally, we consider the
Borel subset $B_0$ of $\der\Omega$ that does not intersect both the
overlap set and the characteristic set. From the measure theoretic
area formula of \cite{valentino1}, joined with the blow-up theorem at
non-characteristic points \cite{valentino2}, we obtain 
\[
 |\nu_\GG|\,\cH^{N-1}_E\restrict B_0= \alpha\, \cS^{Q-1}_\Gamma
\]
for some Borel function $\alpha$ defined on $B_0$, that extends by
zero at points of $\der\Omega\stm B_0$. 
Consequently, \eqref{valentino-identity} implies the identity
\begin{equation}\label{valentino-identity-2}
\int_\Omega \diver_\GG F = \int_{\partial\Omega} 
\left\langle F,\frac{\nu_\GG}{|\nu_\GG|}\right\rangle \,\alpha\,
d\cS^{Q-1}_\Gamma
\end{equation}
for every $C^1$ horizontal vector field $F$ and bounded piecewise
$C^1$ domain $\Omega$ for which the overlap set is a null set for the
boundary measure. The integrand in \eqref{valentino-identity-2} is
undefined on the characteristic set, but this is irrelevant since it
is a null set for the measure. The important fact is that, from
Theorem~5.4 of \cite{valentino2}, we obtain two geometric constants
$c_1,c_2>0$, independent of $\Omega$, such that $c_1\le \alpha\le c_2$
at $\cS^{Q-1}_\Gamma$ a.e. point of $B_0$. Switching from a general
homogeneous distance to the Carnot-Carath\'eodory distance $d$ and
corresponding spherical Hausdorff measure $\cS^{Q-1}$, we obtain   
\begin{equation}\label{divfss}
\begin{split}
\int_{G_m} \di_{\GG,q} \Big(\G (q^{-1} \cdot p) &\anah g_m (q)\Big) \, dq \\
&= \int _{\partial G_m} \G (q^{-1} \cdot p) \Big\langle \anah g_m(q),
\frac{\nu_m (q)}{|\nu_m(q)|} \Big\rangle\, b_m(q) \, d \ss^{Q-1}(q),
\end{split}
\end{equation}
for some $b_m \in L^{\infty}(\ss^{Q-1}\restrict\der G_m)$, where
$\nu_m$ is the non-normalized horizontal normal of $G_m$ and $c_1\le
b_m\le c_2$ at $\cS^{Q-1}$-a.e. point of $\der G_m$ and for every $m$.
(Note that the Radon--Nikodym derivative of $\cS^{Q-1}_\Gamma$ with
respect to $\cS^{Q-1}$, which is bounded away from zero and infinity,
is included in the weight function $b_m$.)

By \eqref{gmmeas2}, $|G_m|\ra 0$, therefore for $p \in D_2 \stm C$,
\begin{equation}\label{zerolim}
\lim_{m \ra \infty}\left| \int_{G_m}\langle \anah \G (p^{-1} \cdot q),
  \anah g_m (q) \rangle \, dq \right| \ra 0,
\end{equation}
since $|\anah g_m|$ is uniformly bounded in $D_2$ and $\anah \G$ is
locally integrable. 

Notice that the signed measures,
\begin{equation}\label{sigmes}
\sigma_m=\Big\langle \anah g_m(\cdot), \frac{\nu_m
  (\cdot)}{|\nu_m(\cdot)|} \Big\rangle\, b_m\,  \ss^{Q-1}
\lfloor \partial G_m,
\end{equation}
have uniformly bounded total variations $\|\sigma_m\|.$
This follows by \eqref{gmmeas2}, as
\begin{equation}\label{smvar}
\begin{split}
\|\sigma_m \| &\leq \| \anah g_m\|_\infty \| b_m
\|_{L^\infty(\ss^{Q-1})}\ss^{Q-1}(\partial G_m) \\  
&\lesssim \sum_j \ss_\G^{Q-1} (\partial U_{m,j})\lesssim \sum_j
r_{m,j}^{Q-1} \\ 
&\lesssim \ss_\G^{Q-1}(C) + \frac{1}{m}.
\end{split}
\end{equation}
Therefore, by a general compactness theorem, see e.g.\ \cite{afp}, we
may extract a weakly converging subsequence $(\sigma_{m_k})_{k \in\N}$
such that $\sigma_{m_k}\ra \sigma$. Furthermore $\spt
\sigma:=\spt|\sigma| \subset C$ and by (\ref{smvar})
\begin{equation}
\label{sigvar}
\|\sigma\| \leq \liminf_{k \ra \infty} \|\sigma_{m_k}\| \lesssim \ss^{Q-1} (C).
\end{equation}

Finally combining (\ref{im1})---(\ref{sigmes}) we get that for $p \in D_2 \stm C$, 
$$
\lim_{k \ra \infty} I_{m_k} (p)= \int_C \G(q^{-1} \cdot p) \, d \sigma(q)
$$
and by (\ref{finconv})---(\ref{imfun})
$$
f(p)=\int_C \G(q^{-1} \cdot p)\,d\sigma(q)+\lim_{k\ra\infty}H_{m_k}(p).
$$
Since the sequence of harmonic functions $(H_{m_k})$ is equicontinuous
on compact subsets of $D_2$, the Arzel\`a-Ascoli theorem implies that
there exists a subsequence $(H_{m_{k_l}})$ which converges uniformly
on compact subsets of $D_2$. From the Mean Value Theorem for
sub-Laplacians and its converse, see \cite{blu}, Theorems 5.5.4 and
5.6.3, we deduce that $(H_{m_{k_l}})$ converges to a function $H$
which is harmonic in $D_2$. Therefore for $p \in D_2 \stm C$,
$$
f(p)=\int_C \G(q^{-1} \cdot p) \, d \sigma q+H (p).
$$
Furthermore the function $H$ is $C^{\infty}$ in $D_2$ with
$\text{Lip}_{\text{B}}(H) \lesssim 1$, therefore by Proposition
\ref{panulip}
$$\|\anah H\|_\infty \lesssim 1.$$
In order to complete the proof it suffices to show that 
\begin{equation}
\label{abscont}
\sigma  \ll \mu \ \text{and} \ h:=\frac{d \sigma}{d \mu} \in L^{\infty} (\mu),
\end{equation}
where $\mu= \ss^{Q-1} \lfloor C$.
The measure-theoretic proof of (\ref{abscont}) can be found in \cite{chousionis-mattila}.
\end{proof}
\begin{lm}
\label{stanfund}
For $p_1,p_2 \neq q \in \hn$
$$
|\G (q^{-1} \cdot p_1)-\G (q^{-1} \cdot p_2)| \lesssim  d(p_1,p_2) (d(q,p_1)^{1-Q}+d(q,p_2)^{1-Q}).
$$ 
\end{lm}
\begin{proof}
Let $p_1,p_2 \neq q \in \hn$. Without loss of generality assume that $d(p_1,q)\leq d(p_2,q)$. We are going to consider two cases.

\

\paragraph{\it Case I} $d(p_1,p_2)\geq \frac{1}{2} d (p_1,q)$. In
this case, since $d_\G$ is globally equivalent to $d$ we have
\begin{equation*}
\begin{split}
|\G (q^{-1} \cdot p_1)-\G (q^{-1} \cdot p_2)| & \lesssim \frac{1}{d(p_1,q)^{Q-2}}+\frac{1}{d(p_2,q)^{Q-2}}\\
\lesssim \frac{ 1}{d(p_1,q)^{Q-2}}&\lesssim \frac{d(p_1,p_2)}{d(p_1,q)^{Q-1}}
\end{split}
\end{equation*}

\paragraph{\it Case II} $d(p_1,p_2)< \frac{1}{2} d (p_1,q)$. In this
case, by the definition of the Carnot-Carath\'eodory metric there
exists a sub-unit curve $\g:[0,d(p_1,p_2)]\ra \hn$ such that $\g
(0)=q^{-1}\cdot p_1$ and $\g (d(p_1,p_2))=q^{-1} \cdot p_2$.
Furthermore, 
\begin{equation}
\label{refe}\g([0,d(p_1,p_2)]) \subset B(q^{-1}\cdot p_1, d (p_1,p_2)).
\end{equation}
Hence for every $t \in [0,d(p_1,p_2)]$
\begin{equation}\label{gammacomp}
\begin{split}
\|\g(t)\|\gtrsim d(0,\g(t)) &\geq d (0, q^{-1}\cdot p_1)-d(\g(t),q^{-1}\cdot p_1) \\
&\geq d(q,p_1)-d(p_1,p_2)\geq \frac{1}{2} d(q,p_1)
\end{split}
\end{equation}
since $d(\g(t),q^{-1}\cdot p_1) \leq d(p_1,p_2)$ by (\ref{refe}).
Therefore, with $T:=d(p_1,p_2)$ we have
\begin{equation*}
\begin{split}
|\G (q^{-1} \cdot p_1)&-\G (q^{-1} \cdot p_2)|=|\G (\g (0))-\G (\g
(T))|=\left| \int_0^T  \frac{d}{dt} (\G (\g(t)) \, dt \right| \\ 
&\leq \int_0^T\left(\sum_{j=1}^m (X_j \G(\g(t)))^2 \right)^\frac{1}{2}
\, dt = \int_0^T|\anah \G (\g(t))| \, dt \\
&\lesssim \int_0^T \frac{dt}{\|\g(t)\|^{Q-1}} \lesssim
\frac{d(p_1,p_2)}{d(p_1,q)^{Q-1}}
\end{split}
\end{equation*}
where we used (\ref{sizeest}) and (\ref{gammacomp}) respectively.
\end{proof}

We are now able to prove Theorem \ref{pos} which as discussed earlier
is also valid for Lipschitz harmonic functions in $\Rn$, with $Q$
replaced by $n$.

\begin{proof}[Proof of Theorem \ref{pos}]
The first statement follows from Theorem \ref{main}. To see this let
$D \supset C$ be a subdomain of $\hn$. Applying Theorem \ref{main} and
recalling that $C$ is a null set for the measure $\cH^{Q-1}$, we
deduce that if $f:D \ra \R$ is Lipschitz in $D$ and $\dh$-harmonic in
$D \stm C$, then there exists an $\dh$-harmonic function $H$ in a
domain $G$, $C\subset G \subset D$, such that
$$
f(p)=H(p) \ \text{for} \ p\in G\stm C.
$$
This implies that $f=H$ in $G$. Hence $f$ is harmonic in $G$, and so
also in $D$. Therefore $C$ is removable.

In order to prove (ii) let $Q-1<s<\dim C$. By Frostman's lemma in
compact metric spaces, see \cite{M}, there exists a nonvanishing Borel measure
$\mu$ with $\spt \mu \subset C$ such that 
$$\mu (B(p,r))\leq r^s \ \text{for} \ p \in \hn, r>0.$$
We define $f:\hn \ra \R^+$ as
$$f(p)= \int \G(q^{-1}\cdot p) d \mu (q).$$
It follows that $f$ is a nonconstant function which is $C^\infty$ in $\hn \stm C$ and $$\dh f=0 \ \text{on} \ \hn \stm C.$$
Furthermore $f$ is Lipschitz. Indeed, for $p_1,p_2 \in \hn$ we may use
Lemma \ref{stanfund} to obtain
\begin{equation*}
\begin{split}
|f(p_1)-f(p_2)|&=\left| \int \G(q^{-1}\cdot p_1) \, d \mu(q) -\int
  \G(q^{-1}\cdot p_2) \, d \mu(q) \right| \\ 
&\lesssim d(p_1,p_2)\left(\int \frac{1}{d(p_1,q)^{Q-1}} \, d\mu(q) +
  \int \frac{1}{d(p_2,q)^{Q-1}} \, d\mu(q) \right) \\ 
&\lesssim d(p_1,p_2).
\end{split}
\end{equation*}
To prove the last inequality let $p \in \hn$, and consider two cases.
If $\dist(p,C)>\diam(C)$,
$$ 
\int \frac{1}{d(p,q)^{Q-1}} \, d \mu(q) \le \frac{
  \mu(C)}{\diam(C)^{Q-1}} \lesssim 1.
$$
If $\dist(p,C)\leq \diam(C)$, then $C \subset B(p,2\diam(C))$. Let
$A=2\diam(C)$, then
\begin{equation*}
\begin{split}
\int \frac{1}{d(p,q)^{Q-1}} \, d \mu(q) & \le \sum_{j=0}^\infty
\int_{B(p,2^{-j}A)\stm B(p,2^{-(j+1)}A)} \frac{d \mu(q)}{d(p,q)^{Q-1}} \\  
& \le \sum_{j=0}^\infty \frac{\mu(B(p,2^{-j}A))}{(2^{-(j+1)}A)^{Q-1}}\\
&\le 2^{Q-1}A^{s-(Q-1)}\sum_{j=0}^\infty (2^{s-(Q-1)})^{-j} \\ &\lesssim 1.
\end{split}
\end{equation*}
Assume, by way of contradiction, that $f$ is $\dh$-harmonic on $\GG$. Since $f \geq 0$, by a Liouville-type theorem for sub-Laplacians, see e.g.\ Theorem 5.8.1 of \cite{blu}, we deduce that $f$ is constant. Hence we have reached a contradiction and consequently $C$ is not removable. 
\end{proof}

\section{Removable sets with positive and finite $\mathcal{H}^{Q-1}$
  measure}\label{sec:rem}

In this section we shall construct a self-similar Cantor set $K$ in
$\GG$ which is $\cL$-removable despite having positive $\cH^{Q-1}$
measure. As noted earlier our proof is rather different than the one
in \cite{chousionis-mattila}. Nevertheless, note that in Theorem
\ref{th-example} there is also one piece $S_0(K)$ of $K$ which is well
separated from the others. This fact allows for a straightforward
application of the condition in Theorem~\ref{unb}.

We let $\Sph := \{ p \in
\GG \, : \, ||p||_{cc} = 1 \}$ be the unit sphere centered at the
origin in this norm. The norm $||\cdot||_{cc}$ is comparable to any
other homogeneous norm on $\GG$, in particular, to the homogeneous norm
$$
|||p|||:=|p_1|+|p_2|^{1/2}+\cdots+|p_s|^{1/s}, \qquad p =
(p_1,p_2,\ldots,p_s).
$$
%We denote by $C_0$ a fixed comparison constant between these two homogeneous norms.

\begin{definition}
For a set $A \subset \Sph$, we define the {\it dilation cone} over $A$
to be the set
$$
\widehat{A} := \{ \delta_r(p) \, : \, r>0, p \in A \}.
$$
\end{definition}

We will prove the following theorem.

\begin{theorem}\label{th-example}
Let $U \subset \Sph$ be a nonempty open set. There exists a
self-similar iterated function system $\cF = \{ S_i : i=0,1,\ldots,M
\}$ with invariant set $K$ such that the following conditions are
satisfied:
\begin{itemize}
\item[(i)] the map $S_0$ has fixed point $0$,
\item[(ii)] $K \subset \widehat{U}$,
\item[(iii)] the pieces $S_0(K),\ldots,S_M(K)$ are pairwise disjoint, and
\item[(iv)] $0<\cH^{Q-1}(K)<\infty$.
\end{itemize}
\end{theorem}

Fix a horizontal vector $\vect \in \fv_1$ and denote by $\V := \{
\exp(t\vect) \, : \, t \in \R \}$ the corresponding horizontal
one-parameter subgroup of $\GG$. Denote by $\W := \exp (\vect^\perp
\times \fv_2 \times \cdots \times \fv_s)$ the corresponding
complementary vertical subgroup, and by $\W_a$, $a \in \V$, the coset
$a * \W$ of $\W$. We may choose $\vect$ and $a$ so that $U \cap \W_a
\ne \emptyset$. In what follows we will assume that $\vect$ and $a$
have been so chosen.

\begin{lemma}\label{lem-example}
There exists a self-similar iterated function system $\cF' = \{ S_i :
i=1,\ldots,M \}$ with invariant set $K'$ such that the following
conditions are satisfied:
\begin{itemize}
\item[(i)] the fixed points of each of the maps $S_i$, $1\le i\le M$,
  lie in $\widehat{U} \cap \W_a$,
\item[(ii)] $K' \subset \widehat{U} \cap \W_a$,
\item[(iii)] the pieces $S_1(K'),\ldots,S_M(K')$ are pairwise disjoint, and
\item[(iv)] $0<\cH^t(K')<\infty$, where $t$ is the Hausdorff dimension of $K'$.
\end{itemize}
\end{lemma}

\begin{remark}\label{metric-remark}
In both Theorem \ref{th-example} and Lemma \ref{lem-example},
condition (iv) follows from condition (iii), by results of Schief, see \cite[Theorem 2.5]{schiefms}. 
\end{remark}

In the proofs we will use the following elementary algebraic fact.

\begin{lemma}\label{elementary-algebraic-fact}
There exists a constant $C_0\ge1$ so that 
\begin{equation}\label{e-a-f-equation}
d(\delta_r(q),q) \le C_0||q||_{cc}
\end{equation}
for all $q \in \GG$ and $0\le r\le 1$.
\end{lemma}

\begin{proof}
By the $1$-homogeneity of both sides of the desired inequality
\eqref{e-a-f-equation}, it suffices to establish the result for points
$q$ with $||q||_{cc}=1$. Since the function $(q,r) \mapsto
d(\delta_r(q),q)$ is continuous from $\GG\times[0,1] \to \R$, the
conclusion follows from compactness of the CC unit sphere.
\end{proof}

%\begin{proof}
%Write $q=(q_1,\ldots,q_s)$. Using the Baker--Campbell--Hausdorff
%formula, we compute
%\begin{equation*}\begin{split}
%\delta_r(q)^{-1} * q &= \bigl( r q_1 , r^2 q_2 , \ldots , r^s q_s
%\bigr)^{-1} * \bigl(q_1 , q_2 , \ldots , q_s \bigr) \\ 
%&= \bigl( (1-r)q_1 , (1-r^2)q_2 , \ldots , (1-r_0^s)q_s \bigr), 
%\end{split}\end{equation*}
%whence
%\begin{equation*}\begin{split}
%d(\delta_r(q),q) &= || \delta_r(q)^{-1} * q ||_{cc} \\ 
%&\le C_0 \left( (1-r)|q_1| + (1-r^{2}) |q_2|^{1/2} + \cdots +
%  (1-r^{s}) |q_s|^{1/s} \right) \\ &\le C_0 \left( |q_1| +
%  |q_2|^{1/2} + \cdots + |q_s|^{1/s} \right) \le C_0^2 ||q||_{cc}.
%\end{split}\end{equation*}
%\end{proof}

\begin{proof}[Proof of Lemma \ref{lem-example}]
We first observe that the coset $\W_a$, equipped with the restriction
of the Carnot-Carath\'eodory metric, is AD $(Q-1)$-regular. This
can be proved in several ways. For instance, we may observe that each
such coset $\W_a$ is isometric to the vertical subgroup $\W$, and that
the Haar measure on $\W$ is AD $(Q-1)$-regular.

Let $B$ be a Carnot-Carath\'eodory ball centered at a point of $U \cap
\W_a$ such that $(1+2C_0)B \subset \widehat{U}$ and $\diam B\le2$, where $C_0$ is as
above. For $\eps>0$, let $p_1,\ldots,p_M \in B \cap \W_a$ be a maximal
collection of points with mutual distance at least $\eps \diam B$. By
the Ahlfors regularity of~$\W_a$, 
\begin{equation}\label{ADRQ-1}
\frac1{C_1} \eps^{1-Q} \le M \le C_1 \eps^{1-Q},
\end{equation}
where $C_1\ge1$ is independent of $\eps$. The choice of $\eps$ will be
made later in the proof, but we note here that we may choose $\eps$
small enough that $M\ge 2^{Q-1}$.

Let $r>0$ be such that
\begin{equation}\label{choice-of-r}
r=
%\min\left\{\frac1{M^{1/(Q-1)}},
%\frac1{20C_0^2\diam B}\eps,
\frac{\eps\diam B}{2C_1^{1/(Q-1)}(10+50C_0(\diam B))}.
%right\}.
\end{equation}
%where $C_0$ is the comparison constant described above.
Note that $r<M^{1/(1-Q)}$ by \eqref{ADRQ-1}. In particular,
$r<\tfrac12$.

We consider the self-similar iterated function system $\cF' = \{S_i \,
: \, i=1,\ldots,M\}$, where $S_i$ is the contraction mapping of $\GG$
with fixed point $p_i$ and contraction ratio $r$. Explicitly,
$S_i:\GG\to\GG$ is given by
$$
S_i(p) = p_i * \delta_r(p_i^{-1} * p), \quad i=1,\ldots,M.
$$
Let $K'$ be the invariant set for $\cF'$. Condition (i) is true by
construction. The inclusion $K' \subset \W_a$ is true since $K'$ is
the closure of the full orbit of the set of fixed points and the coset
$\W_a$ is invariant under each of the maps $S_1,\ldots,S_M$. 

To proceed further we introduce the terminology and notation of
symbolic dynamics. Let $W=\{1,\ldots,M\}$ be the symbol space, let
$W_m$ be the $m$-fold product of $W$ with itself (with $W_0$
containing only the empty set), and let $W_* = \cup_{m\ge 0} W_m$.
Elements of $W_m$ are called {\it words of length $m$} in letters
drawn from $W$. For $w \in W_*$, $w=w_1w_2\cdots w_m$, set $S_w =
S_{w_1} \circ S_{w_2} \circ \cdots \circ S_{w_m}$. 

We will make use of the fact that $K'$ is the closure of the set
$$
\bigcup_{w \in W_*} S_w(p_1);
$$
similarly, for each $i$, $S_i(K')$ is the closure of the set
$$
\bigcup_{w \in W_*} S_{iw}(p_1).
$$

Let $w=w_1w_2\cdots w_m \in W_*$. Repeated application of the triangle
inequality, together with the fact that $S_i$ is a similarity with
contraction ratio $r$, shows that $d(S_w(p_1),p_1)$ is less than or
equal to
$$
d(p_1,S_{w_1}(p_1)) + r \, d(p_1,S_{w_{2}}(p_1))
+ r^2 \, d(p_1,S_{w_{3}}(p_1)) + \cdots + r^{m-1} \,
d(p_1,S_{w_m}(p_1)).
$$
For any $i=1,\ldots,M$,
$$
d(p_1,S_i(p_1)) = d(p_1,p_i*\delta_r(p_i^{-1}*p_1)) 
= d(p_i^{-1}*p_1,\delta_r(p_i^{-1}*p_1)).
$$
Applying Lemma \ref{elementary-algebraic-fact} yields
$$
d(p_1,S_i(p_1)) \le C_0 d(p_1,p_i) \le C_0 \diam B.
$$
Consequently, since $r<\tfrac12$,
$$
d(S_w(p_1),p_1) \le C_0 \frac1{1-r} \diam B \le 2 C_0 \diam B
$$
and so
$$
K' \subset B(p_1,2C_0 \diam B) \subset (1+2C_0)B \subset
\widehat{U}.
$$
This completes the proof of condition (ii). We note in passing that
\begin{equation}\label{diamKbound}
\diam K' \le 4C_0 \diam B.
\end{equation}
In view of Remark \ref{metric-remark}, it remains only to check
condition (iii). Note that the dimension of $K'$, $\log M/\log(1/r)$,
is strictly less than $Q-1$ by the choice of $r$. 

To verify (iii) we show that
\begin{equation}\label{sik-incl}
S_i(K') \subset B(p_i,\tfrac15 \eps)
\end{equation}
for each $i=1,\ldots,M$. (Recall that $d(p_i,p_{i'}) \ge \eps$ for all
$i\ne i'$.) Following a similar argument as above and using
\eqref{diamKbound}, we conclude that
$$
d(S_{iw}(p_1),p_i) = d(S_{iw}(p_1),S_i(p_i)) = rd(S_w(p_1),p_i) \le
(\diam K')r \le 4C_0(\diam B)r.
$$
By the choice of $r$,
$$
4C_0(\diam B)r < 
%(2+10C_0(\diam B))r < 
\frac15\eps.
$$
The proof of \eqref{sik-incl} is complete.
\end{proof}

\begin{remark}\label{dist-sik-sik}
We record the following consequence of \eqref{sik-incl} and the
definition of $\eps$:
$$
\dist(S_i(K'),S_{i'}(K')) \ge \frac35\eps \qquad \mbox{for all $1\le
  i,i'\le M$, $i\ne i'$.}
$$
\end{remark}

\begin{proof}[Proof of Theorem \ref{th-example}]
Let $S_1,\ldots,S_M$ be selected as in the proof of Lemma
\ref{lem-example}, define $r_0>0$ by the equation
$$
r_0^{Q-1} + M r^{Q-1} = 1,
$$
and let $S_0$ be the contraction mapping $S_0(p) = \delta_{r_0}(p)$.
Then condition (i) is satisfied. In view of Remark
\ref{metric-remark}, it suffices to verify conditions (ii) and (iii).

We will employ symbolic dynamics as introduced in the preceding proof
to both iterated function systems $\cF'$ and $\cF$. In order to
distinguish between these two systems, we continue to denote by
$W=\{1,\ldots,M\}$ the word space for the IFS $\cF'$. We let
$V=\{0,\ldots,M\}$ be the symbol space for the IFS $\cF$, we let $V_m$
be the $m$-fold product of $V$ with itself, and we let $V_* =
\cup_{m\ge 0} V_m$. For $v \in V_*$, $v=v_1v_2\cdots v_m$, we set $S_v
= S_{v_1} \circ S_{v_2} \circ \cdots \circ S_{v_m}$. We make use of
the fact that $K$ is the closure of the set
$$
\bigcup_{v \in V_*} S_v(K');
$$
similarly, for each $i=1,\ldots,M$, $S_i(K)$ is the closure of the set
$$
\bigcup_{v \in V_*} S_{iv}(K').
$$
Each element $v \in V_*$ of length $m$ can be uniquely written in the
form
$$
v = u_{k_0}w_{\ell_0}u_{k_1}w_{\ell_1}\cdots u_{k_{T-1}}w_{\ell_{T-1}}u_{k_T}
$$
where $u_k$ is a word consisting of $k$ copies of the letter $0$,
$w_\ell \in W_\ell$, $k_0,\ldots,k_T \ge 0$,
$\ell_0,\ell_1,\ldots,\ell_{T-1} \ge 1$, and
$$
k_0+\ell_0+k_1+\ell_1+\cdots+k_{T-1}+\ell_{T-1}+k_T = m.
$$
Words in $V_*$ with initial letter $i$, in the above representation,
are precisely words for which $k_0=0$ and $w_{\ell_0}$ begins with the
letter $i$. We analyze the image of $K'$ under such words.

For $\delta>0$ and $S \subset \GG$, we denote by $N(S,\delta) = 
\{ p \in \GG \, : \, \dist(p,S) < \delta \}$ the $\delta$-neighborhood
of $S$. 

\begin{lemma}\label{lemma1}
There exists a constant $C>0$ so that if $w \in W_\ell$ and $k,\ell
\in \N$, then 
$$
(S_w \circ S_0^k)(K') \subset N(S_w(K'),r^\ell(1+5C_0(\diam B))).
$$
\end{lemma}

\begin{proof}
Recalling \eqref{diamKbound}, we note that it suffices to prove that
$$
d(S_w(S_0^k(p)),S_w(p)) \le r^\ell (1+5C_0(\diam B)).
$$
for all $p \in K'$. Since $S_w$ has contraction ratio $r^\ell$, this
is equivalent to proving that
$$
d(S_0^k(p),p) \le 1+5C_0(\diam B)
$$
By Lemma \ref{elementary-algebraic-fact},
$$
d(S_0^k(p),p) \le C_0 ||p||_{cc}.
$$
Using the fact that $B\cap K' \ne \emptyset$ and \eqref{diamKbound},
we obtain
$$
||p||_{cc} \le 1 + \diam B + \diam K' \le 1 + (1+4C_0)(\diam B) \le
1 + 5C_0(\diam B),
$$
completing the proof.
\end{proof}

In a geodesic metric space (e.g., $\GG$ equipped with the
Carnot--Carath\'eodory metric), we have
$$
N(N(S,\delta),\eps) = N(S,\delta+\eps) \qquad \mbox{for any set $S$
  and any $\delta,\eps>0$.}
$$
This fact and an easy inductive argument leads to the following
result.

\begin{lemma}\label{lemma2}
If $w_{\ell_0} \in W_{\ell_0}$, $w_{\ell_1} \in W_{\ell_1}$, \dots,
$w_{\ell_{T-1}} \in W_{\ell_{T-1}}$ and $k_1,\ldots,k_T \in \N$, then
$$
(S_{w_{\ell_0}} \circ S_0^{k_1} \circ S_{w_{\ell_1}} \circ S_0^{k_2}
\circ \cdots \circ S_{w_{\ell_{T-1}}} \circ S_0^{k_T})(K')
\subset N(S_{w_{\ell_0}}(K'), \rho)
$$
where
$$
\rho = (r^{\ell_0} + r^{\ell_0+\ell_1} + \cdots +
r^{\ell_0+\ell_1+\cdots+\ell_{T-1}})(1+5C_0(\diam B))).
$$
\end{lemma}

We now conclude the proof of Theorem \ref{th-example}. Since
$r<\tfrac12$ and $\ell_0\ge 1$, we deduce from Lemma \ref{lemma2} that
$$
r^{\ell_0} + r^{\ell_0+\ell_1} + \cdots +
r^{\ell_0+\ell_1+\cdots+\ell_{T-1}} 
\le \frac{r^{\ell_0}}{1-r} \le 2r 
$$
and hence that
$$
(S_{w_{\ell_0}} \circ S_0^{k_1} \circ S_{w_{\ell_1}} \circ S_0^{k_2}
\circ \cdots \circ S_{w_{\ell_{T-1}}} \circ S_0^{k_T})(K') \subset
N(S_{w_{\ell_0}}(K'),2(1+5C_0(\diam B))r).
$$
In particular, if the first letter of $w_{\ell_0}$ is $i$, then
\begin{equation}\label{inclusion-equation}
(S_{w_{\ell_0}} \circ S_0^{k_1} \circ S_{w_{\ell_1}} \circ S_0^{k_2}
\circ \cdots \circ S_{w_{\ell_{T-1}}} \circ S_0^{k_T})(K') \subset
N(S_i(K'),2(1+5C_0(\diam B))r).
\end{equation}
As discussed above, this means that all sets of the form $S_v(K')$,
where $v \in V_*$ has initial letter $i$, are contained in the set on
the right hand side of \eqref{inclusion-equation}, so
$$
S_i(K) \subset N(S_{i}(K'),2(1+5C_0(\diam B))r).
$$
By the choice of $r$, $2(1+5C_0(\diam B))r < \frac15\eps$ and so 
\begin{equation}\label{siksik}
S_i(K) \subset N(S_i(K'),\frac15\eps).
\end{equation}
In view of Remark \ref{dist-sik-sik}, the sets $S_1(K),\ldots,S_M(K)$
are disjoint.

Next, we want to show that $S_0(K) \cap S_i(K) = \emptyset$ for $1\le
i\le M$. To this end, we consider projection $\pi_\V$ into the
horizontal subgroup $\V$. The set $\V$ can be isometrically identified
with $\R$; we denote by $P_\V:\GG\to\R$ the composition of $\pi_\V$
with this identification. There exists a self-similar contraction
$T_i:\R\to\R$ so that $T_i\circ P_\V = P_\V \circ S_i$. Explicitly,
$$
T_0(t) = r_0 t \qquad \mbox{and} \qquad 
T_i(t) = a + r (t-a) \qquad \mbox{for $i=1,\ldots,M$.}
$$
It suffices to prove that
$$
P_\V(S_0(K)) \cap P_\V(S_i(K)) = \emptyset,
$$
i.e.,
$$
T_0(P_\V(K)) \cap T_i(P_\V(K)) = \emptyset.
$$
Since $P_\V(K) \subset [0,a]$, the latter condition holds provided
\begin{equation}\label{r0r}
r_0+r<1.
\end{equation}
Recalling that $r$ and $r_0$ are related by $r_0^{Q-1}+Mr^{Q-1} = 1$,
we rewrite \eqref{r0r} as 
\begin{equation}\label{restriction-on-r}
r < 1 - (1-Mr^{Q-1})^{1/(Q-1)}.
\end{equation}
We observe that
\begin{equation}\label{restriction-on-r-2}
\begin{split}
Mr^{Q-1} 
&\ge \frac{M\eps^{Q-1}(\diam B)^{Q-1}}{2^{Q-1}C_1(10+50C_0\diam
  B)^{Q-1}} \\
&\ge \frac{(\diam B)^{Q-1}}{2^{Q-1}C_1^2(10+50C_0\diam
  B)^{Q-1}}
\end{split}
\end{equation}
In view of \eqref{choice-of-r} and \eqref{restriction-on-r-2}, we see
that \eqref{restriction-on-r} is satisfied provided that
$$
\frac{\eps\diam B}{C_1^{1/(Q-1)}(10+50C_0(\diam B))} <
1 - \left( 1 - \frac{(\diam B)^{Q-1}}{2^{Q-1}C_1^2(10+50C_0\diam
    B)^{Q-1}} \right)^{1/(Q-1)}.
$$
The latter inequality is true provided $\eps$ is chosen sufficiently
small. This completes the proof that $S_0(K)$ is disjoint from each of
the sets $S_i(K)$, $i=1,\ldots,M$, and hence completes the proof of
(iii).

It remains to verify (ii). We first record the identity
$$
K = \bigcup_{k\ge 0} S_0^k \left( \bigcup_{i=1}^M S_i(K)\right) \, .
$$
In view of \eqref{siksik} and the choice of the data,
$S_1(K)\cup\cdots\cup S_M(K) \subset \widehat{U}$. Since $\widehat{U}$
is a dilation cone and $S_0$ is a dilation, it follows that $K\subset
\widehat{U}$ as desired. The proof of the theorem is complete.
\end{proof}

\begin{rem}\label{ssad}
It follows easily, see e.g. \cite[Theorem 2.9]{schiefms} and
\cite[Theorem 5.3.1]{hut}, that if $K$ is the separated set of Theorem
\ref{th-example} the measure $\mathcal{H}^{Q-1} \lfloor K$ is
$(Q-1)$-AD regular.
\end{rem}

In the following we fix some notation. 

\begin{notation}\label{singnot}
For a signed Borel measure $\sigma$ set
$$
T_\sigma (p):=\int k(q^{-1}\cdot p) \, d \sigma (q), \ \text{whenever it
  exists},
$$
$$
T^\e_\sigma (p):=\int_{\hn \stm B(p,\e)} k(q^{-1}\cdot p) \, d \sigma (q)
$$
and
$$
T^\ast_\sigma (p):=\sup_{\e>0}|T^\e_\sigma (p)|.
$$
\end{notation}

The proof of the following lemma is rather similar to that of Lemma 5.4 in \cite{mpa}.

\begin{lm}\label{semmes}
Let $\sigma$ be a signed Borel measure in $\hn$ and $A_\sigma$ a positive constant such that
 $|\sigma|(B(p,r)) \leq A_\sigma r^{Q-1}$ for $ p \in \hn, r>0$. Then 
 $$|T^*_\sigma (p)| \leq \|T_\sigma\|_\infty+A_T \ \text{for} \ p \in \hn,$$
 where $A_T$ is a constant depending only on $\sigma$.
\end{lm}

\begin{proof}
We can assume that $L=\|T_\s\|_\infty< \infty$. The constants that
will appear in the following depend only on $n$ and $\sigma$. For
$\e>0$ and $p \in \hn$,
\begin{equation*}
\begin{split}
\frac{1}{|(B(p,\e / 4))|}&\int_{B(p, \e / 4)} \int_{B(p,
  \e)}\frac{1}{\|q^{-1} \cdot z\|^{Q-1}} \, d |\sigma| (q) \, dz \\
&\approx  \e^{-Q} \int_{B(p, \e / 4)} \int_{B(p, \e)}\frac{1}{\|q^{-1}
  \cdot z\|^{Q-1}} \, d |\sigma| (q) \, dz \\ 
&\leq \e^{-Q} \int_{B(p, \e)} \int_{B(q,2 \e)} \frac{dz}{\|q^{-1}
  \cdot z\|^{Q-1}} \, d|\sigma| (q) \\ 
&\approx  \e^{1-Q} |\s| (B(p,\e))\leq A_\sigma
\end{split}
\end{equation*}
where we used Fubini and the fact that 
$$
\int_{B(q,2 \e)} \frac{dz}{\|q^{-1} \cdot z\|^{Q-1}} \lesssim
\int_{B(q,2 \e)} \frac{dz}{d(q,z)^{Q-1}}\approx  \e,
$$
which is easily checked by summing over the annuli $B(q,2^{1-i} \e)
\stm B(q,2^{-i} \e), i=0,1,\dots$. 

Now because of the inequality established above we can
choose $z \in B(p, \e / 4)$ with $|T_\sigma(z)| \leq L$ such that 
\begin{equation*}
\int_{B(p, \e)}  |k(q^{-1}\cdot z)| \, d|\s|(q) \lesssim \int_{B(p,
  \e)}\frac{1}{\|q^{-1} \cdot z\|^{Q-1}} \, d|\sigma|(q) \leq L_1.
\end{equation*}
Therefore,
\begin{equation*}
\begin{split}
|T_\s^\e(p)-T_\s(z)|&=\left| \int_{\hn \stm B(p, \e)} k(q^{-1}\cdot
  p)d |\s|(q) -\int  k(q^{-1}\cdot z)d |\s|(q)\right| \\ 
&\leq \int_{\hn \stm B(p, \e)} |k(q^{-1}\cdot p)-k(q^{-1}\cdot z)|d
|\s|(q)+ \int_{B(p, \e)}  |k(q^{-1}\cdot z)|d |\s|(q) \\ 
&\leq \int_{\hn \stm B(p, \e)} |k(q^{-1}\cdot p)-k(q^{-1}\cdot z)|d
|\s|(q) +L_1. 
\end{split}
\end{equation*}
Since $k$ is a $C^\infty$, $(1-Q)$-homogeneous function on $\GG
\setminus \{0\}$ by \cite[Proposition 1.7]{fs:hardy} 
\begin{equation}\label{folprop1.7}
|k(X\cdot Y)-k(X)|\leq C \|Y\|_{cc}\|X\|_{cc}^{-Q} \quad \mbox{for all $\|Y\|_{cc}\leq \|X\|_{cc}/2$.}
\end{equation}
Therefore if $z\in B(p,\e / 4)$ and $q \in B(p,\e)^c,$ letting
$X=q^{-1} \cdot z,Y=z^{-1}\cdot p$ we have
that $$\|X\|_{cc}=d(q,z)\geq d(q,p)-d(p,z)>3 \e /4 \geq 3d(z,p)=3
\|Y\|_{cc}$$ 
and
\begin{equation*}
\int_{\hn \stm B(p, \e)} |k(q^{-1}\cdot p)-k(q^{-1}\cdot z)| \, d|\s|(q)
\lesssim \int_{\hn \stm B(p, \e)} \frac{d(p,z)}{d(z,q)^Q } \, d|\s|(q). 
\end{equation*}
Since
\begin{equation*}
\begin{split}
\int_{\hn \stm B(p, \e)} \frac{d(p,z)}{d(z,q)^Q } \, d|\s|(q) & \le
\frac{\e}{2} \sum_{j=0}^\infty \int_{B(z,2^{j} \e) \stm B(z,2^{j-1}
  \e) } \frac{1}{d(p,q)^Q } \, d|\s|(q) \\ 
&\leq  \frac{\e}{2} \sum_{j=0}^\infty  \frac{|\s|(B(p,2^{j} \e))}{(2^{j-1} \e)^Q}  \\
&\leq  A_\s\frac{\e}{2} \sum_{j=0}^\infty  \frac{(2^{j} \e)^{Q-1}}{(2^{j-1} \e)^Q} \\
%&\leq C C_\s 2^{Q-1}\sum_{j=0}^\infty  \frac{2^{j(Q-1)}}{2^{jQ}} \\
%&=C C_\s 2^{Q-1} \sum_{j=0}^\infty  \frac{1}{2^j}\\
&=L_2,
\end{split}
\end{equation*}
we deduce that 
\begin{equation*}
\int_{\hn \stm B(p, \e)} |k(q^{-1}\cdot p)-k(q^{-1}\cdot z)| \, d|\s|(q)
\le L_2.
\end{equation*}
Therefore
$$|T_\s^\e(p)|\leq|T_\s^\e(p)-T_\s(z)|+|T_\s(z)|\leq L_1 +L_2+L.$$
The lemma is proven.
\end{proof}

\begin{thm}\label{unb}
Let $K$ be the separated self similar set obtained in Theorem~\ref{th-example} and let 
$k_i$ be any of the coordinate kernels of $k$. 
If there exists $x=S_w (x) \in K$, $w \in V_\ast$,
such that
$$
\int_{K \setminus S_w(K)} k_i(x^{-1} \cdot y) \, d\cH^{Q-1}(y)\ne 0,
$$
then the maximal operator $T^*_{\mathcal{H}^{Q-1} \lfloor K}$ is
unbounded in $L^2(\mathcal{H}^{Q-1} \lfloor K)$.
\end{thm}

The previous theorem was proved in \cite{chousionis-mattila} in the
abstract setting of complete metric groups with dilations. Here we have
formulated a version tailored to our setting.
We will also need the following Lemma which compares usual maximal
singular integrals to maximal symbolic singular integrals on separated
self-similar sets. The proof can be found in \cite[Lemma
2.4]{chousionis-mattila}.

\begin{lm}\label{compop}
Let $K$ be the separated self similar set obtained in Theorem
\ref{th-example} and let $k_i$ be any of the coordinate kernels of
$k$. Then
\begin{enumerate}
\item there exists a constant $\alpha_K>0$, depending only on the set
  $K$, such that
\begin{equation*}
\dist (S_v(K),K \stm S_v(K)) \geq \alpha_K  \diam (S_v (K))
\end{equation*}
for every $v \in V_\ast$, and
\item there is a constant $A_{\text{C}}$, depending only on the set $K$
  and the kernel $k_i$, such that
\begin{equation*}
\begin{split}
&\left|\int_{S_w(K) \stm S_v(K)} k_i(p^{-1} \cdot y) \, d\cH^{Q-1}(y)
\right| \\
& \quad \quad \le \left| \int_{B(p,2\diam (S_w(K)))\stm B(p,2\diam
    (S_v(K))) } k_i(p^{-1} \cdot y) \, d\cH^{Q-1}(y) \right| + A_{K}
\end{split}
\end{equation*}
for all $w,v \in V$ and $p \in \hn$ for which $S_v (K) \subset S_w(K)$
and $$\dist (p,S_v(K)) \leq \frac{\alpha_K}{2}\diam (S_v(K)).$$
\end{enumerate}
\end{lm}

We can now prove Theorem \ref{unbrem}.

\begin{proof}[Proof of Theorem \ref{unbrem}] 
There exists some $i=1, \dots, m$ such that $k_i$ is not identically
zero in $\GG \stm \{0\}$. Therefore, since $\O_i$ is continuous in
$\Sph$, there exists some open set $U \subset \Sph$ such that, without
loss of generality, $\O_i (p)>0$ for all $p \in U$. In particular
$k_i$ is positive for all $p \in \widehat{U} \setminus \{0\}$.

Now let $K$ be the separated self similar set that we obtain from
Theorem \ref{th-example} for $\widehat{U}$ as above. Notice that since
$K \subset \widehat{U}$, and $k_i$ is positive on $\widehat{U}\stm
\{0\}$
\begin{equation}\label{lastestim}
\begin{split}
\int_{K \stm S_0(K)} k_i(y) \, d\cH^{Q-1}(y) > 0.
\end{split}
\end{equation}
Since $0$ is a fixed point of $K$, Theorem \ref{unb} implies that
$T^\ast_{\ha^{Q-1}\lfloor K}$ is unbounded in $L^2(\ha^{Q-1}\lfloor
K)$.

Suppose that $K$ is not removable. Then there exists a domain
$D\supset K$ and a Lipschitz function $f:D \ra R$ which is
$\dh$-harmonic in $D\stm K$ but not in $D$. By Theorem \ref{main}
there exists a domain $G,K \subset G \subset D$, a Borel function $h:C
\ra \R$ and an $\dh$-harmonic function $H: G \ra \R$ such that 
$$
f(p)=\int_{K} \G(q^{-1}\cdot p)h(q) \, d \ha^{Q-1}(q)+H(p) \ \text{for} \
p \in G \stm K
$$
and $\|h\|_{L^\infty (\ha^{Q-1} \lfloor K)}+\|\anah H \|_\infty
\lesssim 1$. Let $\s=h \ha^{Q-1} \lfloor K$. In this case by the left
invariance of $\anah$ as in (\ref{im1}) and recalling Notation
\ref{singnot} 
$$
T_\s (p)=\anah f(p)- \anah H(p) \ \text{for all} \ p \in G\stm K
$$ 
which implies that
\begin{equation}
\label{tsb1}
|T_\s (p)| \lesssim 1 \ \text{ for all } \ p \in G\stm K. 
\end{equation}
Let $\delta=\dist (K, \hn \stm G )>0$. Then for $p \in \hn \stm G$,
\begin{equation}
\label{tsb2}
|T_\s (p)|\lesssim \int \frac{1}{\|q^{-1} \cdot p\|^{Q-1}} \, d|\s|(q)
\le \frac{|\s|(K)}{\delta^{Q-1}} \lesssim 1.
\end{equation} 
By (\ref{tsb1}) and (\ref{tsb2}) we deduce that $T_\s \in L^\infty$.
Hence, recalling  Remark \ref{ssad},  the measure $\ha^{Q-1} \lfloor
K$ is $(Q-1)$-AD regular and we can apply Lemma \ref{semmes} to
conclude that $T^\ast_\s$ is bounded. 

Furthermore since $f$ is not harmonic in $K$, $h \neq 0$ in a set of
positive $\ha^{Q-1}$ measure. Therefore there exists a point $p \in K$
of approximate continuity (with respect to $\ha^{Q-1} \lfloor K$) of
$h$ such that $h(p) \neq 0$.  Let $w_n \in \{0,1, \dots, M\}^n$, where
$M$ is as in Theorem \ref{th-example}, be such that $p \in S_{w_n}
(K)$. Then by the approximate continuity of $h$,
$$
r^{(1-Q)n}(S_{w_n}^{-1})_\sharp (\s \lfloor S_{w_n}(K))
\rightharpoonup h(p)\ha^{Q-1} \lfloor K \text{ as }  n \ra \infty.
$$ 
We can now check that the boundedness of $T^\ast_\s$ implies that
$T^\ast_{\ha^{Q-1} \lfloor K}$ is bounded. Let $z \in \hn \stm (K\cup
\bigcup_{n=1}^\infty S^{-1}_{w_n}(K))$. If $\dist(z,K)>
\frac{\alpha_{K}}{2}\diam(K)$, then
\begin{equation}
|T_{\ha^{Q-1} \lfloor K}(z)|\lesssim 1. 
\end{equation} 
Therefore we can assume that $\dist(z,K)\leq
\frac{\alpha_{K}}{2}\diam(K)$. Hence for any $w \in V_\ast$,
\begin{equation}
\label{dissw}
\begin{split}
\dist (S_w (z), S_w(K))&=r^{|w|}\dist(z,K) \\
&\leq r^{|w|}\frac{\alpha_{K}}{2} \diam (K) \\
&=\frac{\alpha_{K}}{2} \diam (S_w(K)).
\end{split}
\end{equation}
Notice that the $0$-homogeneity of $k_i$  implies that
$k_i(S^{-1}_{w_n}(q)^{-1}\cdot z)=r^{(Q-1)n} k_i(q^{-1}\cdot
S_{w_n}(z))$. Therefore, 
\begin{equation*}
\begin{split}
h(p)T_{\ha^{Q-1} \lfloor K}(z)&=\lim_{n \ra \infty} r^{(1-Q)n}\int
k_i(q^{-1}\cdot z) \, d(S_{w_n}^{-1})_\sharp (\s \lfloor S_{w_n}(K))(q) \\
&=\lim_{n \ra \infty} r^{(1-Q)n}\int_{S_{w_n}(K)}
k_i(S^{-1}_{w_n}(q)^{-1}\cdot z) \,d\s (q) \\
&=\lim_{n \ra \infty} \int_{S_{w_n}(K)} k_i(q^{-1}\cdot S_{w_n}(z)) \,
d\s (q) \\
&=\lim_{n \ra \infty} \left(\int_{K} k_i(q^{-1}\cdot S_{w_n}(z)) \,
  d\s (q)-\int_{K \stm S_{w_n}(K)} k_i(q^{-1}\cdot S_{w_n}(z)) \, d\s (q)
\right).
\end{split}
\end{equation*}
Since $z \notin \bigcup_{n=1}^\infty S^{-1}_{w_n}(K)$,
\begin{equation*}
\left|\int_{K} k_i(q^{-1}\cdot S_{w_n}(z)) \, d\s (q) \right| \leq
\|T^\ast_\s\|_\infty.
\end{equation*}
Furthermore by Lemma \ref{compop} and (\ref{dissw}) we get that,
\begin{equation*}
\left|\int_{K \stm S_{w_n}(K)} k_i(q^{-1}\cdot S_{w_n}(z)) \, d\s (q)
\right| \leq 2 \|T^\ast_\s\|_\infty+A_{K}.
\end{equation*}
Therefore,
$$
|h(p) T_{\ha^{Q-1} \lfloor K}(z)| \leq 3 \|T^\ast_\s\|_\infty+A_{K},
$$
and since $$\left | K\cup \bigcup_{n=1}^\infty
  S^{-1}_{w_n}(K)\right|=0$$ we get that $T_{\ha^{Q-1} \lfloor
  C_{Q-1}} \in L^\infty$. Hence by Lemma \ref{semmes}
$T^\ast_{\ha^{Q-1} \lfloor K}$ is bounded in $L^2(\ha^{Q-1} \lfloor
K)$ and we have reached a contradiction. The proof of the theorem is
complete.
\end{proof}
   
\begin{rem}\label{verplane} 
Vertical hyperplanes of the form $\{(x,t) \in \hn: x\in W, t\in\R\}$,
where $W$ is a linear hyperplane in $\R^{m}$, are homogeneous
subgroups of $\hn$, that is, they are closed subgroups invariant under
the dilations $\delta_r$. Their Hausdorff dimension is $Q-1$.  
If $V$ is any such vertical hyperplane and $\sigma$ denotes the
$(Q-2)$-dimensional Lebesgue measure on $V$ it follows by
\cite[Theorem 4, page 623 and Corollary 2, page 36]{Ste}
that $T^\ast_\sigma$ is bounded in $L^2(\sigma)$. This   implies, for
example by the methods used in \cite{mpa}, that positive measure
subsets of vertical hyperplanes are not removable for Lipschitz
harmonic functions.
\end{rem}

\section{Concluding comments and questions}

As in the Euclidean case the study of removable sets for Lipschitz
$\cL$-harmonic functions with positive and finite $\ha^{Q-1}$-measure
heavily depends on the study of the singular integral $T(f)=(T_1(f),
\dots, T_m(f))$ where formally 
$$
T_i(f)(p)=\int k_i(p^{-1} \cdot q)f(q) \, d \ha^{Q-1}(q)
$$
and $k=(k_1,\dots,k_m)=\anah \G$. 

Our understanding of such singular integrals is extremely limited even
when the fundamental solution of the sub-Laplacian, and hence the
kernel $k$, have explicit formulas as in the Heisenberg group. There
are two natural directions one could pursue in order to extend our
knowledge of the topic. First of all it is not known what regularity and
smoothness assumptions are needed for a $(Q-1)$-AD regular set $M$ in
order the operator $T$ to be bounded in $L^2(\ha^{Q-1} \lfloor M)$.
Recall that sets which define $L^2$-bounded operators can be seen to
be non-removable, cf.\ Remark \ref{verplane}. Second it is not known
how much we can extend the range of removable $(Q-1)$-dimensional
self-similar sets. We are not aware of any self-similar sets where the
condition in Theorem \ref{unb} fails for all its fixed points.
Nevertheless due to the changes in sign of the kernel  checking that
the integral in Theorem \ref{unb} does not vanish could be technically
very complicated. 

\bibliographystyle{acm}
\bibliography{Carnotlip}
\end{document}